\documentclass{amsart}

\usepackage{amsthm}
\usepackage{enumerate}
\usepackage{amssymb}
\usepackage[utf8]{inputenc}
\usepackage[all,cmtip,2cell]{xy}

\usepackage[colorlinks  = true,
            citecolor   = blue,
            linkcolor   = blue,
            bookmarks   = true,
            linktocpage = true]{hyperref}

\newcommand{\A}[1]{{#1}}

\newcommand{\Cs}{C*}

\newcommand{\bb}[1]{\mathbb{#1}}
\newcommand{\IN}{\bb N}

\newcommand{\dd}{^{**}}

\newcommand{\SOT}{\text{\normalfont\scshape sot}}

\newcommand{\cc}{\subset\hskip-.5em\subset}
\newcommand{\cse}{\prec\hskip-.5em\prec}
\newcommand{\cue}{\sim_{\text{Cu}}}
\newcommand{\cuse}{\precsim_{\text{Cu}}}
\newcommand{\pze}{\sim_{\text{PZ}}}

\newcommand{\cat}[1]{\text{\normalfont\sffamily #1}}

\newcommand{\Cu}{\cat{Cu}}

\newcommand{\seq}[1]{\left\{{#1}_n\right\}_{n\in\IN}}
\newcommand{\norm}[1]{\left\Vert{#1}\right\Vert}

\newcommand{\eps}{\epsilon}

\title{Open projections and suprema in the Cuntz semigroup}

\author{Joan Bosa}
\address{School of Mathematics and Statistics, University of Glasgow, 15 University Gardens, G12 8QW, Glasgow, UK}
\email{joan.bosa@glasgow.ac.uk}

\author{Gabriele Tornetta}
\address{School of Mathematics and Statistics, University of Glasgow, 15 University Gardens, G12 8QW, Glasgow, UK}
\email{g.tornetta.1@research.gla.ac.uk}

\author{Joachim Zacharias}
\address{School of Mathematics and Statistics, University of Glasgow, 15 University Gardens, G12 8QW, Glasgow, UK}
\email{joachim.zacharias@glasgow.ac.uk}

\thanks{\emph{Supported by:}  EPSRC Grant EP/601419/2}

\theoremstyle{plain}
\newtheorem{lemma}{Lemma}[section]

\newtheorem{theorem}[lemma]{Theorem}
\newtheorem{corollary}[lemma]{Corollary}
\newtheorem{proposition}[lemma]{Proposition}
\newtheorem{definition}[lemma]{Definition}
\newtheorem*{proposition*}{Proposition}
\newtheorem*{theorem*}{Theorem}
\newtheorem*{definition*}{Definition}
\newtheorem*{claim*}{Claim}
\newtheorem*{notation*}{Notation}

\newtheorem{remark}[lemma]{Remark}

\newcommand{\K}{\mathcal{K}}
\newcommand{\M}{\mathcal{M}}

\newcommand{\Proj}{{\operatorname{Proj}}}

\newcommand{\W}{{\rm W}}

\begin{document}

\begin{abstract} We provide a new and concise proof of the existence of suprema in the Cuntz semigroup using the open projection picture of the Cuntz semigroup initiated in \cite{ORT}. Our argument is based on the observation that the supremum of a countable set of open projections in the bidual of a \Cs-algebra $A$ is again open and corresponds to the generated hereditary sub-\Cs-algebra of $A$.
\end{abstract}


\maketitle


\section*{Introduction}

The Cuntz semigroup, first introduced in the late 70's by Cuntz (cf. \cite{Cun}, \cite{Cu82}), has over the years emerged as an important tool in the classification of simple \Cs-algebras. Motivated by the possible lack of projections Cuntz defined the semigroup $\W(A)$ as certain  equivalence classes of positive elements in $M_{\infty}(A)$. More recently, the stabilized Cuntz semigroup $\Cu(A)$ given by classes of positive elements in $A \otimes \K$, has been considered. It has a more abstract category-theoretical description put forward in the remarkable work \cite{CEI}, where it is shown that $\Cu(A)$ can be described as equivalence classes of countably generated Hilbert modules. This description is used to establish the existence of suprema in $\Cu(A)$ and the continuity of the natural functor $\Cu(\_)$ from the category of \Cs-algebras to the  category $\Cu$. This fact has turned out to be very important  and has been exploited in many other works, e.g. \cite{ABP,APT,bpt}. The proof in  \cite{CEI} however appears rather involved. An alternative, but still involved proof which is based on the positive element picture of $\Cu(A)$ can be found in \cite{santiago}. But also this proof seems to us to take the reader away from the underlying algebraic structure that is leading to the construction of a suitable representative for the class of the supremum.

Recently a new approach to the Cuntz semigroup $\Cu(A)$ has been proposed in \cite{ORT}  based on the notion of open projections  and a comparison theory for those projections introduced by Peligrad and Zsid\'o \cite{PZ}. Note that in the stable and  separable case there is a natural correspondence between open projections, hereditary subalgebras, countably generated Hilbert modules and positive elements of a given \Cs-algebra.

In this paper we give a proof of the existence of suprema in the Cuntz semigroup of a separable \Cs-algebra based on the open projection picture of $\Cu(A)$ which appears very natural and transparent.  
It stands in between the module picture and the positive element picture, and  provides a constructive proof for one of the main properties of $\Cu(A)$,  complementing the results in \cite{ORT}.  Along the way we observe that for stable algebras every class in the Cuntz semigroup can be represented by a projection in the multiplier algebra.
Essentially, all we need is the very natural concept of compact subequivalence 
for open projections (cf. Definition \ref{compact-subequivalence}) and the fact that increasing strong limits of open projections are open, a fact already known to Akemann  (cf. \cite[Proposition II.5]{Akemann1969}).  However, we observe that if  a family of hereditary subalgebras is directed by inclusion (equivalently, the family of open projections is increasing with respect to the usual order in the positive cone of a \Cs-algebra), then the hereditary subalgebra associated to the limit projection is simply the inductive limit of the system of hereditary subalgebras, where the connecting maps are the natural inclusions (Lemma \ref{lem:her_norm}).

The paper is organized as follows. In the first section we provide some background and well-known results for the Cuntz semigroup, its different  definitions i.e. the positive element, open projection, Hilbert module and hereditary subalgebra picture. In Section 2 we state and prove our results aimed at the construction of suprema of arbitrary sequences of open projections and their relation with the associated hereditary subalgebras. We finish proving the existence of suprema in the Cuntz semigroup in Section 3.


\section{Open Projections and the Cuntz Semigroup}

In this section we briefly recall the definition of the  Cuntz semigroup based on comparison of  positive elements in a \Cs-algebra as well as  alternative descriptions based on Hilbert modules (\cite{CEI}), hereditary subalgebras and corresponding open projections (\cite{ORT}). 
Throughout, we will make the blanket assumption that all \Cs-algebras are separable.

\begin{definition}[Cuntz comparison of positive elements] Let $a, b$ be two positive elements from a \Cs-algebra $A$. We say that $a$ is Cuntz-subequivalent to $b$, in symbols $a\precsim b$, if there exists a sequence $\seq x\subseteq  A$ such that
  $$\norm{x_n^*bx_n - a}\to 0.$$
Cuntz equivalence arises as the antisymmetrization of the above pre-order relation, i.e. $a\sim b$ if and only if $a\precsim b$ and $b\precsim a$. 
\end{definition}
In the commutative setting, the Cuntz equivalence relation just defined reduces to comparison of the support of positive functions (cf. e.g. \cite[Proposition 2.5]{APT}). Hence, equivalence classes are somehow \emph{parametrized} by some open subset of the topological space $X$.

The (stabilized) Cuntz semigroup of a \Cs-algebra $A$ is defined as the set of equivalence classes
  $$\Cu(A):= (A\otimes\K)^+/\sim$$
equipped with the binary Abelian operation $+$ defined by
  $$[a] + [b] := [a\oplus b],$$
whereas the classical Cuntz semigroup is obtained by replacing $A\otimes\K$ by $M_{\infty}(A)$.
It was shown in \cite{CEI} that $\Cu(A)$ belongs to a richer category than just that of Abelian monoids, also denoted by  $\Cu$, and that the natural functor $\Cu(\_)$ from \Cs-algebras to $\Cu$ is sequentially continuous. In fact, it is shown in \cite{APT14} that this functor is continuous, i.e. the property holds for arbitrary inductive limits.

Every positive element $a \in A$ defines the hereditary subalgebra $A_a = \overline{aAa}$ and the Hilbert module $\overline{aA}$. For the class of algebras we consider, every closed hereditary subalgebra and every closed (right)-ideal in $A \otimes \K$ are of this form. The Cuntz semigroup can therefore also be based on equivalence classes of these objects. 


\subsection{Open projections} 
As it is well-known, open subsets of a compact Hausdorff space $X$ can be characterized by (a restatement of) Urysohn's Lemma.
In \cite{Akemann1969}, Akemann used this property to generalize the notion of open subsets to non-commutative \Cs-algebras by replacing sets with projections, and therefore the non-commutative analogue of the above lemma leads naturally to the following.
\begin{definition}[Open projection] Let $A$ be any \Cs-algebra. A projection $p\in A\dd$ is open if it is the strong limit of an increasing net of positive elements $\{a_\alpha\}_{\alpha\in I}\subseteq  A^+$.
\end{definition}

Equivalently, a projection $p\in A\dd$ is open if it belongs to the strong closure of the hereditary subalgebra $ A_p\subseteq A$ (cf. \cite{Akemann1969}), where
  \begin{equation}\label{eq:hereditary}
     A_p := p A\dd p \cap  A = pAp \cap A.
  \end{equation}
In accordance with \cite{ORT}, the set of all open projections in $ A\dd$ will be denoted by 
$P_o( A\dd)$. These projections are in one-to-one correspondence with hereditary subalgebras. 

Recall that the bidual $A\dd$ of a $C^*$-algebra can be identified with the closure in the strong operator topology of $A$ in its universal representation. The von Neumann algebra generated by $A$ in a specific representation is given by projecting onto the representation space. The multiplier algebra $\M (A)$ of  $A$ obtained as  the strict closure of $A$ acting on itself is smaller, since strict convergence implies strong convergence in every representation. In any faithful representation of $A$ the multiplier algebra also acts faithfully which is not the case for $A\dd$ in general. Moreover, the projections in $\M (A)$ are the strict analogues of open projections in the following sense.

\begin{proposition} The projections in $\M (A)$ are those projections in $A\dd$ which are strict limits of  increasing nets of positive elements.
Thus every projection in $\M (A)$ is open, in particular $\Proj(\M(A\otimes\K)) \subseteq P_o((A\otimes\K)^{\dd})$.
\end{proposition}
\begin{proof}
Any projection which is a strict limit of an increasing net in $A$ is in $\M (A)$. On the other hand, if $P \in \M (A)$ is a projection, then $PAP \subseteq A$ is a hereditary subalgebra and any increasing approximate unit of $PAP$ converges strictly to $P$.
\end{proof}

We will see below (Proposition \ref{rep-open-projection-by multiplier}) that if $A$ is stable then every open projection in $A\dd$ is Cuntz equivalent
to one in $\M  (A)$. 

Continuing with the topological analogy, a projection $p\in A\dd$ is said to be \emph{closed} if its complement $1-p\in A\dd$ is an open projection, and so the closure of an open projection can also be defined. To this end, observe that the supremum of an arbitrary set $P$ of open projections in $A\dd$ is still an open projection and, likewise, the infimum of an arbitrary family of closed projections is still a closed projection, by results in \cite{Akemann1969}. Therefore, the closure of an open projection $p\in A\dd$ can be defined as
  $$\overline p := \inf\{q^*q=q\in A\dd\ |\ 1-q\in P_o( A\dd)\ \wedge\ p\leq q\}.$$
Let $ B$ be a sub-\Cs-algebra of $A$. A closed projection $p\in A\dd$ is said to be \emph{compact} in $ B$ if there exists a positive contraction $a\in B^+_1$ such that $pa = p$.  An important relation between open projections is the following
\begin{definition}[Compact containment] Let $p,q\in A\dd$ be two open projections. Then $p$ is said to be compactly contained in $q$ ($p\cc q$ in symbols) if $\overline p$ is compact in $ A_q$.
\end{definition}
A sequence of open projections $\seq p\subseteq P_o( A\dd)$ is said to be \emph{rapidly increasing} if $p_k\cc p_{k+1}$ for any $k\in\IN$. Note that compact containment is quite a restrictive relation for projections, since it requires one projection to be contained in the other. To make it slightly more flexible one can make use of the following definition of an equivalence relation due to Peligrad and Zsid\'o in \cite{PZ}.

\begin{definition}\label{compact-subequivalence}{\rm (PZ-equivalence)} Two open projections $p,q\in P_o( A\dd)$ are said to be PZ-equivalent, $p\pze q$ in symbols, if there exists a partial isometry $v\in A\dd$ such that
  $$p = v^*v,\qquad q = vv^*,$$
and 
  $$v A_p\subseteq A,\qquad v^* A_q\subseteq A.$$
\end{definition}
\begin{remark}
It is clear from this definition that PZ-equivalence is, in general, stronger than Murray-von Neumann equivalence, although there are some cases where the two relations are known to coincide (cf.  \cite{ORT}). 
\end{remark}
Combining compact containment and PZ-subequivalence leads to what we coin compact subequivalence.

\begin{definition}[Compact subequivalence] Two open projections $p,q\in P_o(A\dd)$ are said to be compactly subequivalent, $p\cse q$ in symbols, if there exists $q'\cc q$ such that $p\sim_{\text{PZ}}q'$.
\end{definition}
Observe that the usual compact containment relation $\cc$ is a special instance of compact subequivalence $\cse$.

As presented in \cite{ORT}, $\Cu(A)$ can also be described using open projections. If $p,q$ are open projections in $A\dd$, then $p$ is said to be Cuntz-subequivalent to $q$, in symbols $p\precsim q$, or sometimes also $p\cuse q$, if for every open projection $p'\cc p$ there exists an open projection $q'\cc q$ such that $p'\pze q'$.

As shown by \cite[Theorem 6.1]{ORT} it turns out that Cuntz comparison of positive elements coincides with the Cuntz comparison of the corresponding open support projections, namely
  $$\Cu(A) \cong P_o (( A\otimes \K)\dd)/\cue,$$
as ordered Abelian semigroups. 

The following discussion shows that every class in $\Cu(A)$ can be represented by a projection in $\M (A \otimes \K)$.

For this we need to use the Hilbert module picture for $\Cu(A)$. Recall that every countably generated Hilbert module over $A$ is a submodule of
	$$\ell^2 (A) = \left\{ (a_n) \in A^{\mathbb{N}}\ \Big\vert\ \sum_{n=1}^{\infty} a_n^* a_n  \text{ converges in norm}\right\},$$
and that an open projection $p \in P_o ((A \otimes \K)\dd)$ determines a countably generated Hilbert module over $A$ given by $E_p := p\ell^2(A)  \cap \ell^2(A)$. (Note that $p$ maps $\ell^2(A)$ into $\ell^2(A)\dd$ and $E_p = \{\xi \in \ell^2(A) \mid p\xi =\xi \} \subseteq \ell^2(A)$ which is a closed Hilbert submodule.)  Following \cite{CEI} a Hilbert module $E$ is compactly contained in a Hilbert module $F$, written $E \cc F$, if there exists a positive element $x$ in the compact operators $\K(F)$ on $F$ such that $x \xi =\xi$ for all $\xi \in E$. Moreover $E$ is Cuntz-subequivalent to $F$, written $E \cuse F$, if for every compactly contained Hilbert submodule $E' \cc E$ there exists $F' \cc F$ with $E' \cong F'$. As  mentioned before, $\Cu(A)$ can also be defined as equivalence classes of countably generated Hilbert modules under the equivalence relation $E \cue F$ if $E \cuse F$ and $F \cuse E$. Note that isomorphic Hilbert modules are in particular Cuntz equivalent (cf. \cite[Proposition 4.3]{ORT}).

Similarly, the Cuntz semigroup can be described by classes of hereditary subalgebras of $A \otimes \K$.  The hereditary subalgebra of $A\otimes \K$ corresponding to a countably generated Hilbert module $E$ is given by $\K(E)$. Conversely, given $p \in P_o((A\otimes\K)^{**})$ on checks that $(A \otimes \K)_p$ is isomorphic to $\K (E_p)$. 

By Kasparov's stabilization theorem we know that 
	$$E_p \oplus \ell^2 (A) \cong \ell^2 (A).$$
Let $P$ be the projection onto $E_p$ in the above orthogonal decomposition. Then $P$ is a projection in ${\mathcal M}(A \otimes \K)$ with $(A\otimes \K)_P = \K (E_p)$, thus the original open projection $p$ is Cuntz equivalent to the projection $P$ that belongs to ${\mathcal M}(A \otimes \K)$. Moreover $I-P$ is also open with $(A\otimes \K)_{I-P} \cong  A\otimes \K$. 

\begin{proposition}\label{rep-open-projection-by multiplier} Every class in $\Cu(A)$ has a representative in the set of projections in $\M (A \otimes \K)$, so that $\Cu(A)$ can also be thought of as the set of Cuntz equivalence classes of projections from  $\M(A\otimes\K)$.
\end{proposition}
Thus 
$$\Cu(A)\cong\Proj(\M(A\otimes\K))/\sim_\Cu$$ 
with the Cuntz subequivalence relation we had in $P_o((A\otimes\K)^{**})$, which a priori involves open projections not necessarily in $\M(A \otimes \K)$. 

By virtue of this last proposition we will often use a capital $P$ for a projection in $\M(A\otimes \K)$, given as orthogonal projection onto $E_p$ in $E_p \oplus \ell^2(A)=\ell^2(A)$ corresponding to the open projection $p \in P_o((A\otimes\K)^{**})$. Corollary \ref{dilate3} shows that this replacement is compatible with PZ-subequivalence.

We finish this section by providing the main result that stems from our new equivalent characterization of $\Cu(A)$, Corollary \ref{dilate3}, which shows the existence of a unitary element in $\M(A\otimes\K)$ that implements the Cuntz subequivalence between two open projections. This is the analogue of the crucial result \cite[Proposition 2.4]{Ror92}  for stable algebras, where the existence of such unitary is shown under the stable rank one assumption. Before that, we need the following slight refinement of Kasparov's stabilization theorem. For a set $S \subseteq \mathbb{N} $ let
$$\ell^2 (S,A) = \left\{ (a_n) \in A^{S}\ \big\vert\ \sum_{n \in S} a_n^* a_n  \text{ converges in norm}\right\}.$$
\begin{lemma}\label{dilate1} Let $E \subseteq F $ be an inclusion of countably generated Hilbert modules over $A$ and $S \subseteq \mathbb{N} $ infinite with infinite complement. Then there exists a unitary $U: \ell^2(\mathbb{N},A) \to \ell^2(\mathbb{N},A) \oplus F$ such that $U_{ | \ell^2(S,A)} $ is a unitary mapping onto $\ell^2(S,A) \oplus E$.
\end{lemma}
\begin{proof} This is an easy modification of the standard proof of Kasparov's stabilization Theorem (e.g. \cite[Theorem 13.6.2]{Bla}) which we briefly indicate for convenience.  By adjoining a unit we may assume that $A$ is unital, so that $\ell^2(\mathbb{N},A)$ has the canonical basis $(e_n)$. Let $(\eta_n) $ be a bounded sequence of generators of $F$ (e.g. a dense sequence in the unit ball) such that every sequence member appears infinitely often, and let $(\eta_j)_{j \in S}$ generating $E$ such that each generator appears infinitely often too. Then the polar decomposition $U |T|$ of $T: \ell^2(\mathbb{N} ,A) \to \ell^2(\mathbb{N} ,A ) \oplus F$ given by $T=\sum_n 2^{-n} \Theta_{\eta_n + 2^{-n} e_n , e_n}$ provides the required unitary $U$. (Here $\Theta_{\xi,\eta}$ denotes the `rank 1 operator' $\Theta_{\xi,\eta}\zeta= \xi \langle \eta, \zeta\rangle$.)
\end{proof}
Note that we cannot show with this construction that $E$ is complemented in $F$ (which would be false in general), since the projections onto $F$ and $\ell^2(S,A)$ do not commute in general.
\begin{corollary}\label{dilate2} Let $E$, $F$ be countably generated Hilbert modules and $v \in  {\mathcal L}_A ( E, F)$ an isometry. Then there exists a unitary $u \in \ {\mathcal L}_A(E \oplus \ell^2(A) , F\oplus \ell^2(A))  = {\mathcal L}_A (\ell^2(A))=\M (A \otimes \K)$ extending $v$.
\end{corollary}
\begin{proof} We may assume that $v$ is an inclusion. By Lemma \ref{dilate1} we can extend an inclusion to the canonical  inclusion $\ell^2(S,A) \hookrightarrow \ell^2(\mathbb{N},A)$, corresponding to the inclusion $S \hookrightarrow \mathbb{N}$. By adding another copy of $\ell^2(\mathbb{N} , A)$ we can extend this inclusion  to the required unitary of $\ell^2(A)$.
\end{proof} 
\begin{corollary}\label{dilate3} Let $p, q \in P_o(( A\otimes \K)\dd)$ with  $p\cse q$, i.e. there exist $q'\cc q$ and a partial isometry $v\in A\dd$ such that $p = v^*v, q' = vv^*,$
and 
$v A_p\subseteq A,v^* A_{q'}\subseteq A$ and let $P$ and $Q'$ be the corresponding projections onto $E_p$ and $E_{q'}$.
Then there exists a unitary $u \in \M (A \otimes \K)$ extending $v$ such that 
$uP u^* = Q' $.  
\end{corollary}
\begin{proof}
We may regard $v$ as an isometry from the Hilbert module $E_p$ into $E_{q}$ and apply  \ref{dilate2}.
\end{proof}
\begin{remark}
With the same notation as in Corollary \ref{dilate3}, if $p\precsim q$, then for all $P'\cc P$ there exists a unitary $u\in\M(A\otimes\K)$ such that $uP'u^*\cc Q$.
\end{remark}

Note that a similar statement applies to sequences of open projections satisfying either $p_1 \cse p_2 \cse p_3 \cse \ldots \,\,$ or $p_1 \precsim p_2 \precsim p_3 \precsim \ldots \,\,$.

\section{Hereditary Subalgebras and Open Projections}

In this section we establish the hereditary sub-\Cs-algebra analogue of the operation of taking suprema of countably many open projections in $P_o(A\dd)$. We start by observing that, given two open projections $p,q\in A^{**}$ such that $p\leq q$ (as positive elements), then $q$ obviously acts as a unit on $p$, and $A_p\subseteq A_q$ (cf. \cite[\S4.5]{ORT}). This property will be extensively used throughout this paper.  The following very natural  Lemma might be well-known to experts. Since we have not been able to find a proof in the literature we provide one.

\begin{lemma}\label{lem:her_norm} Let $\seq p$ be an increasing sequence of open projections in $ A^{**}$. Then
  	$$ A_p = \overline{\bigcup_{k\in\IN} A_{p_k}},$$
	where $p:=\SOT\lim_{n\to\infty}p_n$.
\end{lemma}
\begin{proof} Let $B$ denote the inductive limit on the right side which coincides with the union. By construction $B$ is a hereditary subalgebra of $A$, and therefore, there exists a generator $a \in B$ such that $B = \overline{a Aa}$ (recall that we assume $A$ to be separable). It is then enough to show that the support projection $q\in A^{**}$ of $a$ coincides with $p$. Let $\seq a$ be a sequence of positive elements converging to $a$ in norm such that $a_n\in A_{p_n}$ for any $n\in\IN$. Let $q$ be the support projection of $a$ and $q_n$ be the support projection of $a_n$ for any $n\in\IN$. It is clear that $q_n\leq p_n\leq q$ for any $n\in\IN$ from which it follows that
    $$\sup\seq q \leq \SOT\lim_{n\to\infty}p_n\leq q.$$
  Now suppose that $q'$ is an open projection such that $q_n\leq q'$ for any $n\in\IN$. This implies that $a_nq'=q'a_n = a_n$ for any $n\in\IN$ and so
	$aq'=q'a=a$. Therefore, $q\leq q'$, which leads to $q = \sup\seq q$, whence $p = q$.
\end{proof}
A similar result, using the positive element picture, can be found in \cite[Lemma 4.2]{bpt}, where this result is used to prove that the Cuntz semigroup of any \Cs-algebra of stable rank 1 admits suprema. 

\begin{corollary} Let $\seq p$ be an increasing sequence of open projections in $ A^{**}$. Then
		$$\overline{ A_p}^\SOT = \overline{\bigcup_{k\in\IN}\overline{ A_{p_k}}^\SOT}^\SOT,$$
	where $p:=\SOT\lim_{n\to\infty}p_n$.
\end{corollary}
\begin{proof} From Lemma \ref{lem:her_norm} one has
    $\overline{ A_p}^\SOT = \overline{\bigcup_{n\in\IN} A_{p_n}}^\SOT.$
  Therefore, by using that
    $\bigcup_{n\in\IN}\overline{ A_{p_n}}^\SOT\subseteq\overline{\bigcup_{n\in\IN} A_{p_n}}^\SOT  \text{ and }\hspace{0.5cm}A_p \subseteq \overline{\bigcup_{n\in\IN}\overline{ A_{p_n}}^\SOT}^\SOT,$
    it follows that
    \[ A_p \subseteq \overline{\bigcup_{n\in\IN}\overline{ A_{p_n}}^\SOT}^\SOT\subseteq\overline{\bigcup_{n\in\IN} A_{p_n}}^\SOT=\overline{ A_p}^\SOT.\qedhere\]
\end{proof}  
The result that now follows is an example of an application of Lemma \ref{lem:her_norm}. The construction of the supremum, i.e. the join of an arbitrary family of projections in the bidual $A\dd$ of a \Cs-algebra $A$ can be carried out by relying on the lattice structure on the set of projections in $A\dd$. In the case of an \emph{increasing} sequence of projections, Lemma \ref{lem:her_norm} shows that the hereditary sub-\Cs-algebra associated to the supremum coincides with the inductive limit of the increasing sequence of hereditary sub-\Cs-algebras associated to each projection in the considered subset of $P_o( A\dd)$. For the general case we then have the following
\begin{proposition} Let $\seq p\subseteq P_o( A\dd)$ be an arbitrary sequence of open projections in $ A\dd$, and let $p:=\sup\seq p$. Then
    $$ A_p = \bigvee_{n\in\IN} A_{p_n},$$
  i.e. $ A_p$ coincides with the hereditary sub-\Cs-algebra of $A$ generated by the family of hereditary sub-\Cs-algebras $\{ A_{p_n}\ |\ n\in\IN\}$.
\end{proposition}
\begin{proof} Consider the new sequence of open projections $\seq q$ defined by
    $q_1 := p_1,\,\,q_{n+1} := q_{n} \vee p_{n+1},\,\,\,\forall n\in\IN.$
  This clearly defines an increasing sequence of open projections, and moreover
    $p:=\sup\seq p = \SOT\lim_{n\to\infty}q_n.$
  Therefore, using Lemma \ref{lem:her_norm}, one has the identification
    $ A_p = \overline{\bigcup_{k\in\IN} A_{q_k}}.$
  
  By definition, $ A_{p_k}$ is clearly contained in $ A_{q_k}$ for any $k\in\IN$, so
    $\bigvee_{k\in\IN} A_{p_k}\subseteq\overline{\bigcup_{k\in\IN} A_{q_k}}.$
  On the other hand, $ A_{q_k}$ is contained in $\bigvee_{n\in\{1,\ldots, k\}} A_{p_n}$, so 
    $\overline{\bigcup_{k\in\IN} A_{q_k}}\subseteq\bigvee_{k\in\IN} A_{p_k},$
  which shows equality.    
\end{proof}


\section{Suprema in the Cuntz semigroup}
In this section we show that the existence of suprema in the stabilised Cuntz semigroup can be proven by just referring to the open projection picture, using the results discussed in the previous sections.

\begin{lemma}\label{lem:main} Let $p$ be the strong limit of an increasing sequence of open projections $p_1\leq p_2\leq\cdots$. Then, for every $q\cc p$, there is an $n\in\IN$ and an open projection $q'\cc p_n$ such that $q\sim_{PZ}q'$.
\end{lemma}
\begin{proof} By the definition of the relation $q\cc p$ there exists a positive element $a$ in the unit ball of $ A_p$ such that $\overline qa=\overline q$, and by the same argument as in \cite{CEI} (cf. \cite[Proposition 4.11]{APT}), one can find $a'\in$ \Cs$(a)$ such that $\overline q(a'-\eps)_+=\overline q$. 

Let $a_n\in A_{p_n}$ be such that $\norm{a_n-a'}<\eps$, which exists by Lemma \ref{lem:her_norm}. By \cite[Lemma 2.2]{KR} there is a contraction $d\in A_p$ such that $da_nd^* = (a'-\eps)_+$, and it follows by \cite[Theorem 1.4]{PZ} that  $$\overline q\leq p_{x^*x}\sim_{PZ} p_{xx^*}\leq p_n,$$ where $x = a_n^{1/2}d^*$.
   
Since $\leq$ and $\sim_{PZ}$ are special instances of $\precsim_\Cu$ and $\sim_\Cu$ respectively, using \cite[Proposition 4.10]{ORT} one also has
  $$q\cc p_{x^*x}\sim_\Cu p_{xx^*}\precsim_\Cu p_n.$$
Therefore there must exist an open projection $q'\cc p_n$ such that $q\sim_{PZ}q'$.
\end{proof}

\begin{proposition}\label{prop:containment} If $p_1\cc p_2\cc\cdots$ is a rapidly increasing sequence of open projections in $P_o( A^{**})$, then $p_1\leq p_2\leq\cdots$ and
    $$\sup[p_n]=[\SOT\lim_{n\to\infty}p_n].$$
\end{proposition}
\begin{proof} Let $p$ be the strong limit of the $p_n$s and suppose that $[q]$ is such that $[p_n]\leq[q]$ for any $n\in\IN$. By Lemma \ref{lem:main}, for every $p'\cc p$ there is an $n\in\IN$ and an open projection $q'$ such that $p'\sim_{PZ}q'\cc p_n$. But, since $[p_n]\leq[q]$, there exists a $q''\cc q$ such that $q''\sim_{\text{PZ}}q'$. Therefore, $[p]\leq[q]$. Since $[q]$ is arbitrary, it follows that $[p]=\sup[p_n]$.
\end{proof}

To prove that every increasing sequence, in the Cuntz sense, has a supremum in the open projection picture one of course needs a more general result. If one na\"ively tries to tackle this problem inside $A$ directly, one runs into the following problem. Let $\seq p$ by any sequence of open projections in $P_o( A\dd)$ with the property that $p_n\cse p_{n+1}$ for every $n\in\IN$. By assumption there are open projections $\seq q$ such that $q_k\cc p_{k+1}$ and $p_k\sim_{\text{PZ}}q_k$. These determine an inductive sequence $(A_{p_k},\phi_k)_{k\in\IN}$ of hereditary subalgebras of $A$, where the connecting maps are given by the adjoint action of partial isometries $\seq v$ satisfying $p_k = v_k^*v_k$, $q_k = v_kv_k^*$ and $v A_{p_k}\subseteq A$, $v^* A_{q_k}\subseteq A$, i.e. $\phi_k(a) = v_k^*av_k$ for any $a\in A_{p_k}$. Denoting by $\tilde A$ the inductive limit of such a sequence, one gets maps $\seq{\rho}$ that make the following diagram
  $$\xymatrix{%
     A_{p_k}\ar[r]^{\phi_k}\ar[dr]_{\rho_k}&  A_{p_{k+1}}\ar[d]^{\rho_{k+1}}\\
    & \tilde{ A}%
  }$$
commutative. But unless $\tilde A\subseteq A$, nothing more can be said about this sequence to conclude the existence of the supremum of $\seq p$. Indeed, using \cite[Example 1]{RT}, one may show that  $\tilde{ A}$ is not always a subalgebra of $A$. To see this, let $p,q$ be the corresponding open projections associated to the two Cuntz equivalent Hilbert modules, which do not embed one into the other, described in \cite[Example 1]{RT}. Without loss of generality assume that $p$ is the unit of $A^{**}$. Now, choose two rapidly increasing sequences of open projections $(p_i)$ and $(q_i)$ that converge to $p$ and $q$ respectively, and isomorphisms $\phi_i:A_{p_i}\to A_{q_i}$. Composing $\phi_i$ with the embedding $\iota_i:A_{q_i}\to A_{q_{i+1}}$ and $\phi_{i+1}^{-1}$, we get a $\phi_{i+1}^{-1}\circ\iota_i\circ\phi_i:A_{p_i}\to A_{p_{i+1}}$. Taking the inductive limit with these maps, we get an algebra isomorphic to $A_q$ which however does not embed in $A_p=A$, though all $A_{p_i}$ are hereditary subalgebras of $A$.

On the other hand, by working with $A\otimes \K$ instead of $A$, one can extend the above partial isometries into unitaries (Corollary \ref{dilate3}), in order to construct a \emph{tower} rather than a \emph{tunnel}.

\begin{lemma}\label{lem:rapidly} Every sequence $\seq p$ of open projections in $P_o(\A A\otimes \K)^{**}$ with the property that $p_n\cse p_{n+1}$ for every $n\in\IN$ has a supremum in $\Cu(\A A)$.
\end{lemma}
\begin{proof} Denote by $q_k$ the element that satisfies $p_{k-1}\sim_{PZ}q_k\cc p_k$ coming from the definition of the relation $p_k\cse p_{k+1}$, and by capital letters (e.g. $P_k,Q_k$) the Cuntz equivalent projections in $\M(A\otimes\K)$. By Corollary \ref{dilate3}, there exist a collection of unitaries $\seq u$ such that  $u_{k-1}P_{k-1} u_{k-1}^*=Q_k$ for all $k\in\IN$. Hence, $P_{k-1}=u_{k-1}^*Q_ku_{k-1}\cc u_{k-1}^*P_ku_{k-1}$ and therefore one has that
	$$P_1\cc u_1^*P_2u_1\cc u_1^*u_2^*P_3u_2u_1\cc u_1^*u_2^*u_3^*P_4u_3u_2u_1\cc \cdots.$$
Denoting by $\overline{P_i}:= (\prod^{n-1}_{i=1}u_i)^*P_n(\prod^{n-1}_{i=1}u_i)$, let $P:=\SOT\lim_{n\to\infty}\overline{P_i}$. 
By Proposition \ref{prop:containment} it follows that $[P]=\sup[\overline{P}_n]$ which implies that $[P]=\sup[p_i]$ since $[p_i]=[P_i]=[\overline{P_i}]$.
\end{proof}
\begin{remark}
The above could also be proven from the hereditary subalgebras point of view. In this case, using the same notation as in the above proof, one has that
	$$A_{P_1}\subseteq u_1^*A_{P_2}u_1\subseteq u_1^*u_2^*A_{P_3}u_2u_1\subseteq\ldots\subseteq(\prod^{n-1}_{i=1}u_i)^*A_{P_n}(\prod^{n-1}_{i=1}u_i)\subseteq\ldots\hspace{0.2cm},$$
where they belong to $A\otimes \K$ since $\seq u$ are unitaries in $\M(A\otimes \K)$ and $A_{P_n}$ are hereditary subalgebras of $A\otimes \K$.

Denoting by $P$ the open projection associated to the hereditary subalgebra
	$$A_P=\overline{\bigcup^\infty_{n=1}(\prod^{n-1}_{i=1}u_i)^*A_{P_n}(\prod^{n-1}_{i=1}u_i)},$$
it follows that $[P]=\sup[P_n]$.
\end{remark}
The above is an intermediate step towards the more general proof of the existence of suprema for arbitrary Cuntz-increasing sequences of open projections in $\A A\otimes \K$.

\begin{theorem}\label{thm:supofcuinc} Every Cuntz-increasing sequence $\seq p$ of projections in $P_o(\A A\otimes \K)^{**}$ admits a supremum in $\Cu(\A A)$.
\end{theorem}
\begin{proof} Without loss of generality we may assume that $\A A$ is a stable \Cs-algebra. By assumptions there are positive contractions $\{a_{n,k}\}_{n,k\in\IN}\subseteq A^+_1$ such that
    $p_n = \SOT\lim_{k\to\infty} a_{n,k}$
  and $a_{n,k}\leq a_{n,k+1}$ for any $k,n\in\IN$.
  
   These elements can be modified to yield rapidly increasing sequences of positive elements by setting
    $$a'_{n,k} := \left(a_{n,k} - \tfrac1k\right)_+.$$
  Denoting by $q_{n,k}$ the support projections associated to these new elements $a'_{n,k}$, one has
  $q_{n,k}\cc q_{n,k+1}$ for any $k,n\in\IN$. Now, starting with e.g. $q_{1,1}$ and applying Lemma \ref{lem:main} to $q_{1,1}\cc q_{1,2}\cc p_1\precsim p_2$, one gets $m_1\in\IN$ and $q_{1,1}\cc p_{2,m_1}$ such that $q_{1,1}\sim_\text{PZ}q'_{1,1}$. By iterating  these steps one can construct a sequence of open projections $q_k:=q_{k,m_{k-1}}$ that satisfies
  	$$q_1\sim_{\text{PZ}}q'_{1,1}\cc q_2\sim_{\text{PZ}}q'_{2,m_1}\cc q_3\cdots,$$
  i.e.
  	$$q_1\cse q_2\cse q_3\cse q_4\cse\cdots.$$
  
Observe that $[q_k]\leq [p_k]$ for any $k\in\IN$, and that for any $n,m\in \IN$ there exists $l\in \IN$ such that $[q_{n,m}]\leq [q_l]$. Therefore,
    $[p_n]\leq\sup_k[q_k]\leq\sup_k[p_n],$
  which implies
    $$\sup_n[p_n]\leq\sup_k[q_k]\leq\sup_n[p_n], \,\,\,\text{ i.e. }\,\,\,\sup_n[p_n] = \sup_k[q_k].$$
  The existence of the supremum follows from Lemma \ref{lem:rapidly}. 
\end{proof}
\emph{En passant} we observe that we have the following realization for suprema in the Cuntz semigroup.
\begin{corollary}\label{cor:supassot} Every element $x\in \Cu(\A A)$ can be written as the Cuntz class of the strict limit of an increasing sequence of projections in $\M(\A A\otimes \K)$.
\end{corollary}


\section*{Acknowledgements} The authors would like to express their gratitude to Hannes Thiel, Aaron Tikuisis and Stuart White for many invaluable comments and feedback.

\bibliography{refs}{}

\begin{thebibliography}{10}

\bibitem{Akemann1969}
C.~A. Akemann.
\newblock {The general Stone-Weierstrass problem}.
\newblock {\em {J. of Funct. Anal.}}, 4(2):277--294, 1969.

\bibitem{ABP}
R.~Antoine, J.~Bosa, and F.~Perera.
\newblock Completions of monoids with applications to the {C}untz semigroup.
\newblock {\em Internat. J. Math.}, 22(6):837--861, 2011.

\bibitem{APT14}
R.~Antoine, F.~Perera, and H.~Thiel.
\newblock Tensor products and regularity properties of {C}untz semigroups.
\newblock {\em Preprint}.

\bibitem{APT}
P.~Ara, F.~Perera, and A.~S. Toms.
\newblock {$K$}-theory for operator algebras. {C}lassification of
  {$C^*$}-algebras.
\newblock In {\em Aspects of operator algebras and applications}, volume 534 of
  {\em Contemp. Math.}, pages 1--71. Amer. Math. Soc., Providence, RI, 2011.

\bibitem{Bla}
Bruce Blackadar.
\newblock {\em {$K$}-theory for operator algebras}, volume~5 of {\em
  Mathematical Sciences Research Institute Publications}.
\newblock Cambridge University Press, Cambridge, second edition, 1998.

\bibitem{bpt}
N.~Brown, F.~Perera, and A.~S. Toms.
\newblock {The Cuntz semigroup, the Elliott conjecture, and dimension functions
  on C*-algebras}.
\newblock {\em {J. Reine Angew. Math}}, 621:191--211, 2008.

\bibitem{CEI}
K.~T. Coward, G.~A. Elliott, and C.~Ivanescu.
\newblock {The Cuntz semigroup as an invariant for \Cs-algebras.}
\newblock {\em {J. Reine Angew. Math. }}, 623:161--193, 2008.

\bibitem{Cun}
J.~Cuntz.
\newblock {Dimension functions on simple C*-algebras}.
\newblock {\em {Math. Ann}}, 233:145--153, 1978.

\bibitem{Cu82}
J.~Cuntz.
\newblock The internal structure of simple {$C^{\ast} $}-algebras.
\newblock In {\em Operator algebras and applications, {P}art {I} ({K}ingston,
  {O}nt., 1980)}, volume~38 of {\em Proc. Sympos. Pure Math.}, pages 85--115.
  Amer. Math. Soc., Providence, R.I., 1982.

\bibitem{KR}
E.~Kirchberg and M.~R{\o}rdam.
\newblock {Infinite Non-simple \Cs-Algebras: Absorbing the Cuntz Algebra
  $\mathcal O_\infty$}.
\newblock {\em {Adv. in Math. }}, 167(2):195--264, 2002.

\bibitem{ORT}
E.~Ortega, M.~R{\o}rdam, and H.~Thiel.
\newblock {The Cuntz semigroup and comparison of open projections }.
\newblock {\em {J. of Funct. Anal. }}, 260(12):3474--3493, 2011.

\bibitem{PZ}
C.~Peligrad and L.~Zsid\'o.
\newblock {Open projection of $C^*$-algebras comparison and regularity.}
\newblock In {\em {Operator theoretical methods. Proceedings of the 17th
  international conference on operator theory (Timi\c soara, Romania, June
  23-26, 1998)}}, pages 285--300, {Bucharest}, 2000. {The Theta Foundation}.

\bibitem{RT}
Leonel Robert and Aaron Tikuisis.
\newblock Hilbert {$C^*$}-modules over a commutative {$C^*$}-algebra.
\newblock {\em Proc. Lond. Math. Soc. (3)}, 102(2):229--256, 2011.

\bibitem{Ror92}
Mikael R{\o}rdam.
\newblock On the structure of simple {$C^*$}-algebras tensored with a
  {UHF}-algebra. {II}.
\newblock {\em J. Funct. Anal.}, 107(2):255--269, 1992.

\bibitem{santiago}
L.~Santiago.
\newblock {\em {Classification of non-simple \Cs-algebras: Inductive limits of
  splitting interval algebras}}.
\newblock PhD thesis, {Department of Mathematics, University of Toronto},
  {Toronto, Canada}, 2008.

\end{thebibliography}
\bibliographystyle{plain}

\end{document}